\documentclass[11pt]{amsart}

\usepackage{amsmath,amssymb,latexsym,soul,cite}
\usepackage{xcolor,enumitem,graphicx}
\usepackage[colorlinks=true,urlcolor=blue,
citecolor=red,linkcolor=blue,linktocpage,pdfpagelabels,
bookmarksnumbered,bookmarksopen]{hyperref}
\usepackage[english]{babel}
\usepackage[T1]{fontenc}
\usepackage[utf8]{inputenc}
\usepackage[hyperpageref]{backref}

\allowdisplaybreaks
\usepackage{mathtools}
\mathtoolsset{showonlyrefs}

\usepackage[left=2.8cm,right=2.8cm,top=2.9cm,bottom=2.9cm]{geometry}
\numberwithin{equation}{section}

\newtheorem{thm}{Theorem}[section]

  \theoremstyle{plain}
  \newtheorem{lem}[thm]{Lemma}
  \theoremstyle{plain}
  \newtheorem{prop}[thm]{Proposition}
  \theoremstyle{plain}
  \newtheorem{cor}[thm]{Corollary}
  \theoremstyle{plain}
  \newtheorem{definition}[thm]{Definition}
    \theoremstyle{definition}
\newtheorem{rem}[thm]{Remark}

\newcommand{\R}{{\mathbb R}}
\newcommand{\N}{{\mathbb N}}

\title[Symmetry of fractional Neumann eigenfunctions in the ball]{Symmetry of fractional Neumann eigenfunctions\\ in the ball}

\author[V.~Bobkov]{Vladimir Bobkov}
\author[E.~Parini]{Enea Parini}

\address[V.~Bobkov]{Institute of Mathematics, Ufa Federal Research Centre, RAS, 
Chernyshevsky str. 112, 450008 Ufa, Russia} \email{bobkov@matem.anrb.ru}
\address[E.~Parini]{Aix Marseille Univ, CNRS, I2M, 3 place Victor Hugo, 13331 Marseille CEDEX 3, France}
\email{enea.parini@univ-amu.fr}



\thanks{E.P.~acknowledges partial support from the project ``ANR STOIQUES - Shape and Topology Optimization: Impactful QUestions and Emerging Subjects'' (ANR-24-CE40-2216) by the French National Research Agency.}

\begin{document}

\begin{abstract}
	We investigate symmetry properties of the first nontrivial eigenfunctions of the fractional Laplacian $(-\Delta)^s$, where $s \in (0,1)$, in an $N$-dimensional ball with nonlocal Neumann boundary conditions.
	By means of a spectral stability result, we prove that, when $s$ is sufficiently close to $1$, the eigenspace associated to the first nontrivial eigenvalue is generated by $N$ antisymmetric eigenfunctions with exactly two nodal domains in the ball. 
\end{abstract}

\maketitle

\section{Introduction}\label{sec:intro}

Let $s \in (0,1)$ and let $\Omega \subset \R^N$ be a bounded Lipschitz domain, $N \geq 2$. 
We consider the eigenvalue problem
\begin{equation}\label{eq:eigenvalueproblem0} 
	\left\{ \begin{array}{r c l l} (-\Delta)^s u & = & \mu \, u & \text{in }\Omega, \\ \mathcal N_s u & = & 0 & \text{in }\R^N \setminus \overline{\Omega},\end{array} \right.
\end{equation}
where $(-\Delta)^s$ is the fractional Laplace operator defined pointwise, for sufficiently smooth functions, as
\begin{equation}\label{eq:frac-lap}
(-\Delta)^s u(x) = 
c_{N,s}
\lim_{\varepsilon \to 0^+} \int_{\mathbb{R}^N \setminus B_\varepsilon(x)} \frac{u(x)-u(y)}{|x-y|^{N+2s}} \, dy,
\end{equation}
where $c_{N,s}$ is an appropriate normalization constant (see, e.g., \cite[Section~3]{dinezzapalatuccivaldinoci}), and where
\begin{equation}\label{eq:N} 
(\mathcal N_s u)(x) 
:= 
c_{N,s}
\int_\Omega \frac{u(x)-u(y)}{|x-y|^{N+2s}}\,dy, 
\quad
x \in \Omega^c := \R^N \setminus \overline{\Omega}.
\end{equation}

The relation $\mathcal N_s u=0$ in $\Omega^c$ is a nonlocal counterpart of the classical Neumann boundary conditions. 
This concept was first introduced in \cite{dipierrorosotonvaldinoci}, where its probabilistic interpretation was also provided. One of the advantages of this notion is the validity of nonlocal analogs of the usual integration by parts formulae, such as
\[ 
\int_{\Omega} (-\Delta)^s u \,dx = - \int_{\Omega^c} \mathcal{N}_s u \,dx,
\]
for sufficiently smooth functions, which generalizes the well-known relation
\[ \int_{\Omega} \Delta u \,dx = \int_{\partial \Omega} \frac{\partial u}{\partial \nu} \,dS,\]
where $\frac{\partial u}{\partial \nu} = \nabla u \cdot \nu$ is the normal derivative of $u$, and $\nu$ is the outward unit normal vector. 

The validity of the integration by parts formulae makes the setting particularly suitable for the application of variational methods. The authors of \cite{dipierrorosotonvaldinoci} introduced a Hilbert space $H^s_{\Omega,0}$ consisting of measurable functions $u:\mathbb{R}^N \to \mathbb{R}$ and endowed with the norm
\[ 
u \mapsto 
\left( 
\int_\Omega u^2  \,dx
\right)^{\frac{1}{2}}
+
\left( 
\iint_{\R^{2N}\setminus (\Omega^c)^2}
 \frac{|u(x)-u(y)|^2}{|x-y|^{N+2s}}\,dx\,dy
\right)^{\frac{1}{2}},
\]
and showed that the operator $(-\Delta)^s$ defined on $H^s_{\Omega,0}$ satisfies the assumptions of the spectral theorem. In particular, all eigenvalues of \eqref{eq:eigenvalueproblem0} form a discrete monotone sequence $\{\mu_n^{(s)}\}_{n \in \mathbb{N}}$ diverging to infinity:
$$
0 = \mu_0^{(s)} < \mu_1^{(s)} \leq \dots \leq \mu_n^{(s)} \to +\infty, \quad n \to +\infty.
$$ 
When necessary, we will stress the dependence on the domain by using the expanded notation $\mu_n^{(s)}(\Omega)$.
We note, for clarity, that we follow the notation  $\mathbb{N} := \{0,1,2,\dots\}$ and $\mathbb{N}^* := \{1,2,3,\dots\}$.
 
The first nontrivial eigenvalue $\mu_1^{(s)}$ 
has the following variational characterization:
\begin{equation}\label{eq:mu1}
\mu_1^{(s)} = 
\inf 
\left\{
\frac{c_{N,s}}{2}
\iint_{\R^{2N} \setminus (\Omega^c)^2} \frac{|u(x)-u(y)|^2}{|x-y|^{N+2s}}\,dx\,dy \,\bigg|\, u \in H^s_{\Omega,0}, ~\int_\Omega u^2 \,dx = 1,~ \int_\Omega u \,dx = 0\right\}, 
\end{equation}
and any corresponding minimizer $u$ is a weak solution of the problem \eqref{eq:eigenvalueproblem0} with $\mu=\mu_1^{(s)}$, that is,
\begin{equation}\label{eq:weak}
	\frac{c_{N,s}}{2}
\iint_{\R^{2N} \setminus (\Omega^c)^2} \frac{(u(x)-u(y))(v(x)-v(y))}{|x-y|^{N+2s}}\,dx\,dy	= \mu_1^{(s)} \int_\Omega u v \,dx 
	\quad \text{for any}~ v \in H^s_{\Omega,0}.
\end{equation}
Hereinafter, we will formally refer to the case $s=1$ as to the \textit{local} Neumann Laplacian in $\Omega$ and denote by $\{\mu_n^{(1)}\}_{n \in \N}$ the corresponding spectrum. 
The exact connection between $\{\mu_n^{(s)}\}_{n\in \N}$ and $\{\mu_n^{(1)}\}_{n \in \N}$ will be established in Section~\ref{sec:convergence}.

\medskip
The aim of the present work is to investigate symmetry properties of the first nontrivial eigenfunctions of the fractional Laplacian under the nonlocal Neumann boundary conditions. In particular, we are interested in the following question:
\begin{enumerate}
\item[$(\mathcal{Q})$]\label{question:1} 
Let $\Omega$ be a ball $B$. 
Is any first nontrivial eigenfunction antisymmetric with respect to a hyperplane passing through the center of $B$?
\end{enumerate}

For the \textit{local} Neumann Laplacian, the question is well known to have an affirmative answer. 
One possible proof is to note that all corresponding eigenvalues are precisely nonnegative roots of the function \[\Theta_{l}(r) := \frac{d}{d r} \left(r^{\frac{2-N}{2}}J_{\frac{1}{2}(2l+N-2)}(r)\right),\] where $J_\nu$ is the $\nu$-th Bessel function of the first kind, $l \in \mathbb{N}$, and that there is a basis of eigenfunctions with separable geometry.
Then it is sufficient to compare the first {positive} zero of $\Theta_{0}$ (which is the first positive radial eigenvalue) and the first {positive} zero of $\Theta_{1}$ (which is the first positive antisymmetric eigenvalue), and show that the former is strictly greater than the latter, see \cite[Lemma~2.6]{helffersundqvist}. 
However, in the nonlocal setting,  the lack of closed form solutions of the problem \eqref{eq:eigenvalueproblem0} makes it difficult to apply the same arguments. 

Another way to answer the question $(\mathcal{Q})$ in the local case is the following: if we suppose that some first nontrivial eigenfunction $u$ is radial, then we have $\mu_1^{(1)}(B)=\lambda_1^{(1)}(B_r)$, where $\lambda_1^{(1)}(B_r)$ is the first eigenvalue of the local Dirichlet Laplacian in a ball $B_r \subset B$, which is a nodal domain of $u$.
On the other hand, by \cite{Fr} and the domain monotonicity, respectively, we have
\begin{equation}\label{eq:friedland}
\mu_1^{(1)}(B) < \lambda_1^{(1)}(B) < \lambda_1^{(1)}(B_r), 
\end{equation}
which gives a contradiction. 
In the nonlocal setting, in order to go through the same path, it would be sufficient to have the inequalities $\lambda_{1}^{(s)}(B_r) \leq \mu_1^{(s)}(B)$ and $\mu_1^{(s)}(B) \leq \lambda_{1}^{(s)}(B)$, since the strict domain monotonicity of the type $\lambda_{1}^{(s)}(B) < \lambda_{1}^{(s)}(B_r)$ holds true. 
However, it does not seem immediate to understand whether these two inequalities are satisfied, in general.

\medskip

In the present work, we provide partial results on the question $(\mathcal{Q})$. First, we prove the following structure theorem for the eigenspace associated to the first nontrivial eigenvalue.

\begin{thm}\label{thm:main1}
	Let $B \subset \R^N$ be a ball centered at the origin. 
	Let $s \in (0,1)$. 
	Then the following dichotomy occurs:
	\begin{itemize}
		\item Either all eigenfunctions associated to $\mu_1^{(s)}$ are radially symmetric, or
		\item There exists $k \in \N$ such that the eigenspace associated to $\mu_1^{(s)}$ is spanned by $k$ radial eigenfunctions and $N$ eigenfunctions $v_1,\dots,v_N$, where each $v_i$ is antisymmetric with respect to the coordinate hyperplane $\{x_i=0\}$ and has exactly two nodal domains in $B$. 
	\end{itemize} 
\end{thm}

Thanks to Theorem \ref{thm:main1}, and exploiting the stability of the spectrum of the fractional Neumann Laplacian as $s \to 1^-$ (see Theorem~\ref{thm:stable} below), we can deduce the following symmetry result.

\begin{thm}\label{thm:main2}
	Let $B \subset \R^N$ be a ball centered at the origin. 
	Then there exists $s^* \in (0,1)$ such that for each $s \in (s^*,1)$ the eigenspace associated to $\mu_1^{(s)}$ is spanned by $N$ eigenfunctions $v_1,\dots, v_N$, where each $v_i$ is antisymmetric with respect to the coordinate hyperplane $\{x_i = 0\}$ and has exactly two nodal domains in $B$. 
\end{thm}

It remains an interesting open problem to determine whether the result of Theorem \ref{thm:main2} actually holds true for every $s \in (0,1)$. 

\medskip
The paper is organized as follows. 
In Section~\ref{sec:prelim}, we introduce a few notions on the functional framework and discuss a  variational characterization of eigenvalues. 
In Section~\ref{sec:convergence}, we prove the convergence of the nonlocal Neumann eigenvalues towards the local Neumann eigenvalues as $s \to 1^-$.
Section~\ref{sec:symmetry} is devoted to the symmetry analysis of the first nontrivial eigenfunctions and to the proof of Theorems~\ref{thm:main1} and \ref{thm:main2}.

\section{Notation and preliminary results}\label{sec:prelim}

\subsection{Functional setting}

As above, let $s \in (0,1)$, $N \geq 2$, and $\Omega \subset \R^N$ be a bounded Lipschitz domain.  
Recall the definition of the Hilbert space $H^s_{\Omega,0}$ from \cite{dipierrorosotonvaldinoci}: 
\[ 
H^s_{\Omega,0} 
= 
\left\{ u:\R^N \to \R \,\Big|\, u \text{ measurable},\, \|u\|_{L^2(\Omega)} + [u]_{H^s_{\Omega,0}} < +\infty \right\},\]
where 
$$
\|u\|_{L^2(\Omega)}
=
\left(
\int_\Omega u^2 \,dx
\right)^{\frac{1}{2}}
\quad \text{and} \quad 
[u]_{H_{\Omega,0}^s}
=
\left(
\iint_{\R^{2N}\setminus (\Omega^c)^2} \frac{|u(x)-u(y)|^2}{|x-y|^{N+2s}}\,dx\,dy \right)^{\frac{1}{2}}.
$$
Observe that
\begin{equation}\label{eq:decomp:1}
	[u]_{H_{\Omega,0}^s}^2
	= \int_{\Omega}\int_\Omega 
	\frac{|u(x)-u(y)|^2}{|x-y|^{N+2s}}\,dx\,dy + \,2 \int_{\Omega} \int_{\Omega^c} \frac{|u(x)-u(y)|^2}{|x-y|^{N+2s}}\,dx\,dy.
\end{equation}
We will also denote by $H^s(\Omega)$ the standard fractional Sobolev space with the Gagliardo seminorm $[\cdot]_{H^s(\Omega)}$ defined as
$$
[u]_{H^s(\Omega)}
=
\left(
\int_{\Omega}
\int_{\Omega} \frac{|u(x)-u(y)|^2}{|x-y|^{N+2s}}\,dx\,dy \right)^{\frac{1}{2}}.
$$

\begin{rem}\label{rem:compact}
By \eqref{eq:decomp:1}, the restriction of any function from $H_{\Omega,0}^s$ to $\Omega$ belongs to $H^s(\Omega)$. 
In particular, since $H^s(\Omega)$ is embedded in $L^2(\Omega)$ compactly (see, e.g, \cite[Theorem~7.1]{dinezzapalatuccivaldinoci}), 
we get the same property for restrictions of functions from $H_{\Omega,0}^s$.
\end{rem}

The relation between $H_{\Omega,0}^s$ and $H^s(\Omega)$ noted in Remark~\ref{rem:compact} is complemented by the following general result. 
\begin{prop}\label{prop:minimalextension}
For any $u \in H^s(\Omega)$, there exists a unique function $\widetilde{u}: \mathbb{R}^N \to \mathbb{R}$ such that  $\widetilde{u} \in H^s_{\Omega,0}$, $\widetilde{u} = u$ in $\Omega$, and
\[ 
[\widetilde{u}]_{H^s_{\Omega,0}} \leq [v]_{H^s_{\Omega,0}}
\]
for every $v \in H^s_{\Omega,0}$ satisfying $v=u$ in $\Omega$. 
Moreover, $\mathcal{N}_s \widetilde{u} = 0$ in $\Omega^c$.
\end{prop}
\begin{proof}
By \cite[Theorem~5.4]{dinezzapalatuccivaldinoci}, any $u \in H^s(\Omega)$ admits an extension $\hat{u} \in H^s(\mathbb{R}^N)$ such that 
$$
\|\hat{u}\|_{L^2(\mathbb{R}^N)} 
+
[\hat{u}]_{H^s(\mathbb{R}^N)}
\leq
C
(\|u\|_{L^2(\Omega)} 
+
[u]_{H^s(\Omega)}),
$$
where $C>0$ is a constant independent from $u$. 
Noting that $[\hat{u}]_{H^s_{\Omega,0}} \leq [\hat{u}]_{H^s(\mathbb{R}^N)}$, we conclude that $\hat{u} \in H^s_{\Omega,0}$. 
Let us consider the minimization problem
\begin{equation}\label{eq:minimization1} 
\inf \left\{ [v]_{H^s_{\Omega,0}}^2 \,\Big|\, v \in H^s_{\Omega,0},\,v=u \text{ in } \Omega\right\}.
\end{equation}
By means of the direct method of the Calculus of Variations and Remark~\ref{rem:compact}, 
this problem admits a minimizer $\widetilde{u} \in H^s_{\Omega,0}$, which is unique by the strict convexity of the functional. 
Therefore, $\widetilde{u}=u$ in $\Omega$, and $[\widetilde{u}]^2_{H^s_{\Omega,0}} \leq [v]_{H^s_{\Omega,0}}^2$ for every $v \in H^s_{\Omega,0}$ satisfying $v=u$ in $\Omega$. 

Let us now prove that $\mathcal{N}_s \widetilde{u} = 0$ in $\Omega^c$.
Let $t \in \R$ and $\varphi \in C_c(\Omega^c)$. 
It is not hard to see from \eqref{eq:decomp:1} that $\varphi \in H^s_{\Omega,0}$, and we evidently have $\widetilde{u}+t \varphi = u$ in $\Omega$. 
By the minimality of $\widetilde{u}$,
\[ \frac{d}{dt}\left([\widetilde{u}+t\varphi]^2_{H^s_{\Omega,0}} \right)_{|t=0} = 0,\]
which translates to
\begin{equation}\label{eq:esten1}  
\iint_{(\R^{2N}) \setminus (\Omega^c)^2} \frac{(\widetilde{u}(x)-\widetilde{u}(y))(\varphi(x)-\varphi(y))}{|x-y|^{N+2s}}\,dx\,dy = 0.
\end{equation}
Notice that
$$
\iint_{(\R^{2N}) \setminus (\Omega^c)^2} \frac{(\widetilde{u}(x)-\widetilde{u}(y))(\varphi(x)-\varphi(y))}{|x-y|^{N+2s}}\,dx\,dy
=
2\int_{\Omega} \int_{\Omega^c} \frac{(\widetilde{u}(x)-\widetilde{u}(y))\varphi(x)}{|x-y|^{N+2s}}\,dx\,dy.
$$
Recalling that $\widetilde{u}, \varphi \in H^s_{\Omega,0}$, we apply H\"older's inequality to observe that $\frac{(\widetilde{u}(x)-\widetilde{u}(y))\varphi(x)}{|x-y|^{N+2s}} \in L^1(\Omega \times \Omega^c)$. 
Thus, Fubini's theorem, the definition of $\mathcal{N}_s \widetilde{u}$, and the equality \eqref{eq:esten1} yield
\[ \int_{\Omega^c} (\mathcal{N}_s \widetilde{u}) \varphi \,dx = 0.\]
Since $\varphi$ is arbitrary, we conclude that $\mathcal{N}_s \widetilde{u} = 0$ in $\Omega^c$.
\end{proof}

\begin{rem}
In general, by the definition \eqref{eq:N} of $\mathcal{N}_s$, 
for any measurable function $u: \Omega \to \mathbb{R}$ with $\int_\Omega u \,dx <+\infty$, a unique function $\widetilde{u}: \mathbb{R}^N \to \mathbb{R}$ satisfying $\widetilde{u} = u$ in $\Omega$ and 
$\mathcal{N}_s \widetilde{u} = 0$ in $\Omega^c$ can be defined in $\Omega^c$ via the formula
\begin{equation}\label{eq:u-outside2} 
	\widetilde{u}(x) = \frac{\displaystyle \int_\Omega \frac{u(y)}{|x-y|^{N+2s}}\,dy}{\displaystyle \int_\Omega \frac{1}{|x-y|^{N+2s}}\,dy},
\end{equation}
cf.\ \cite[Remark~3.12]{dipierrorosotonvaldinoci}.
\end{rem}

For a given $u \in H^s(\Omega)$, we will call the function $\widetilde{u}$ provided by Proposition~\ref{prop:minimalextension} the \emph{$s$-minimal extension} of $u$. We define the following subset of $H^s_{\Omega,0}$:
$$
\widetilde{H}^s_{\Omega,0} 
:= 
\left\{ 
u \in H^s_{\Omega,0}\,|\, \mathcal{N}_s u = 0 \text{ in } \Omega^c
\right\}.
$$

\begin{prop}\label{prop:closedsubspace}
$\widetilde{H}^s_{\Omega,0}$ is a closed subspace of $H^s_{\Omega,0}$, and hence a Hilbert space.
\end{prop}
\begin{proof}
Since the mapping $u \mapsto \mathcal{N}_s u$ is linear, we see that $\widetilde{H}^s_{\Omega,0}$ is indeed a subspace of $H^s_{\Omega,0}$. 
Let us show that it is closed.
Let $\{u_k\}_{k \in \N^*}$ be a sequence in $\widetilde{H}^s_{\Omega,0}$ converging in $H^s_{\Omega,0}$ to a function $u \in H^s_{\Omega,0}$. 
In particular, for any $\varphi \in C_c(\Omega^c)$ we apply the H\"older inequality to obtain 
$$
 \iint_{\R^{2N} \setminus (\Omega^c)^2} \frac{(u_k(x)-u_k(y))(\varphi(x)-\varphi(y))}{|x-y|^{N+2s}}\,dx\,dy 
 \to 
 \iint_{\R^{2N} \setminus (\Omega^c)^2} \frac{(u(x)-u(y))(\varphi(x)-\varphi(y))}{|x-y|^{N+2s}}\,dx\,dy
$$
as $k \to +\infty$. 
As in the proof of Proposition~\ref{prop:minimalextension}, we use Fubini's theorem to get
\begin{equation}\label{eq:closed} 
\iint_{\R^{2N} \setminus (\Omega^c)^2} \frac{(u_k(x)-u_k(y))(\varphi(x)-\varphi(y))}{|x-y|^{N+2s}}\,dx\,dy 
=
2\int_{\Omega^c} (\mathcal{N}_s u_k) \varphi \,dx 
\end{equation}
for every $u_k$, and the same arguments give the equality \eqref{eq:closed} also for $u$. 
Therefore, since $u_k \in \widetilde{H}^s_{\Omega,0}$, we infer from the convergence that 
\[ \int_{\Omega^c} (\mathcal{N}_s u) \varphi \,dx = 0.\]
Thanks to the arbitrariness of $\varphi \in C_c(\Omega^c)$, we get $u \in \widetilde{H}^s_{\Omega,0}$.
\end{proof}	

\begin{rem}\label{rem:identification}
For a fixed $s \in (0,1)$, 
Proposition \ref{prop:minimalextension} allows to identify a function from $H^s(\Omega)$ with its (unique) $s$-minimal extension, which belongs to $\widetilde{H}^s_{\Omega,0}$; conversely, every function in $\widetilde{H}^s_{\Omega,0}$ can be uniquely associated with its restriction to $\Omega$, which belongs to $H^s(\Omega)$. 
In other words, the spaces $H^s(\Omega)$ and  $\widetilde{H}^s_{\Omega,0}$ can be equivalently characterized as
$$
H^s(\Omega)
=
\left\{
u|_{\Omega} \,\Big|\, u \in \widetilde{H}^s_{\Omega,0}
\right\}
\quad \text{and} \quad 
\widetilde{H}^s_{\Omega,0}
=
\left\{
\widetilde{u} \,|\, u \in H^s(\Omega)
\right\}.
$$
Occasionally, we will switch between these two points of view, when no ambiguity arises.
\end{rem}

Let us introduce a subspace of $L^2(\Omega)$ consisting of functions with zero mean in $\Omega$:
\begin{align}
L^2_0(\Omega) 
&:= 
\left\{u \in L^2(\Omega)\,\bigg|\, \int_\Omega u \,dx = 0\right\}. 
\end{align}
Hereinafter, we write, for example,  $\widetilde{H}^s_{\Omega,0} \cap L^2_0(\Omega)$, meaning that the restriction of any function from $\widetilde{H}^s_{\Omega,0}$ to $\Omega$ belongs to $L^2_0(\Omega)$. 
Thanks to \cite[Lemma~3.10]{dipierrorosotonvaldinoci} and Proposition~\ref{prop:closedsubspace}, $H^s_{\Omega,0} \cap L^2_0(\Omega)$ and $\widetilde{H}^s_{\Omega,0} \cap L^2_0(\Omega)$ are closed subspaces of $H^s_{\Omega,0}$, and hence they are Hilbert spaces on their own.

\subsection{Neumann eigenvalue problem}

Recall that, according to \cite[Theorem 3.11]{dipierrorosotonvaldinoci}, there exist a discrete monotone sequence $\{\mu_n^{(s)}\}_{n \in \N}$ of eigenvalues of the fractional Neumann Laplacian in $\Omega$ satisfying 
$$
0 = \mu_0^{(s)} < \mu_1^{(s)} \leq \dots \leq \mu_n^{(s)} \to +\infty, \quad n \to +\infty, 
$$ 
and a sequence of corresponding eigenfunctions $\{u_n^{(s)}\}_{n \in \N} \subset H^s_{\Omega,0}$ which are mutually orthogonal in $L^2(\Omega)$. 
It is not hard to see that nonzero constant functions are eigenfunctions corresponding to $\mu_0^{(s)} = 0$.
This implies that any higher eigenfunction $u_n^{(s)}$  has zero mean in $\Omega$.
Moreover, each $u_n^{(s)}$ also satisfies $\mathcal{N}_s u_n^{(s)} = 0$ in $\Omega^c$, so that $u_n^{(s)} \in \widetilde{H}^s_{\Omega,0} \cap L^2_0(\Omega)$, cf.\ \cite[Theorem~2.8]{ML1}.

Eigenfunctions $u_n^{(s)}$ have some additional regularity properties. By \cite[Proposition~3.4]{ML1}, they are bounded in $\R^N$. As a consequence, \cite[Theorem~1.1]{AFNRS} implies that each $u_n^{(s)} \in C^\alpha(\overline{\Omega})$ for some $\alpha \in (0,1)$. The result of \cite[Proposition~5.2]{dipierrorosotonvaldinoci} then yields $u_n^{(s)} \in C(\mathbb{R}^N)$.

\medskip
By standard spectral theory and Remark~\ref{rem:compact}, all nontrivial eigenvalues of the fractional Neumann Laplacian in $\Omega$ can be characterized by the relation
\begin{equation}\label{eq:muk-s}
	\mu_n^{(s)} 
	= 
	\min_{A\in M_n} 
	\max_{u \in A \setminus \{0\}} \frac{\frac{c_{N,s}}{2} \, [u]_{H_{\Omega,0}^s}^2}{\|u\|_{L^2(\Omega)}^2},
\quad n \geq 1,
\end{equation} 
where $M_n$ is the collections of $n$-dimensional subspaces of $H_{\Omega,0}^s \cap L^2_0(\Omega)$. 
In much the same way, all nontrivial eigenvalues of the local Neumann Laplacian in $\Omega$ have the following characterization:
\begin{equation}\label{eq:muk-1}
	\mu_n^{(1)}
	= 
	\min_{A\in Y_n} 
	\max_{u \in A\setminus \{0\}} \frac{\|\nabla u\|_{L^2(\Omega)}^2}{\|u\|_{L^2(\Omega)}^2},
\quad n \geq 1,
\end{equation} 
where $Y_n$ is the collections of $n$-dimensional subspaces of $H^1(\Omega) \cap L^2_0(\Omega)$. 

The following proposition shows that it is possible to characterize eigenvalues as critical values over the Hilbert space $\widetilde{H}^s_{\Omega,0} \cap L^2_0(\Omega)$.

\begin{prop}\label{prop:courant-fisher}
	For any $n \in \N^*$, let 
	\[ \tilde{\mu}_n^{(s)} := \min_{A \in \widetilde{M}_n} \max_{u \in A} \frac{\frac{c_{N,s}}{2}  [u]_{H^s_{\Omega,0}}^2}{\|u\|_{L^2(\Omega)}^2},\]
	where $\widetilde{M}_n$ is the collection of $n$-dimensional subspaces of $\widetilde{H}^s_{\Omega,0} \cap L^2_0(\Omega)$.
	Then $\tilde{\mu}_n^{(s)} = \mu_n^{(s)}$. 
\end{prop}
\begin{proof}
	Since $\widetilde{M}_n \subset M_n$, it trivially holds $\mu_n^{(s)} \leq \widetilde{\mu}_n^{(s)}$. 
	Let $\widetilde{A} = \{u_1^{(s)},\dots,u_n^{(s)}\}$.
	Since each $u_n^{(s)}$ belongs to $\widetilde{H}^s_{\Omega,0} \cap L^2_0(\Omega)$ and all eigenfunctions are mutually orthogonal in $L^2(\Omega)$, we have $\widetilde{A} \in \widetilde{M}_n$.
	Therefore,
	\[ 
	\widetilde{\mu}_n^{(s)} \leq \max_{u \in \widetilde{A}} \frac{\frac{c_{N,s}}{2}  [u]_{H^s_{\Omega,0}}^2}{\|u\|_{L^2(\Omega)}^2} = \mu_n^{(s)},
	\]
	and we conclude that $\widetilde{\mu}_n^{(s)} = \mu_n^{(s)}$.
\end{proof}

As a corollary of Proposition~\ref{prop:courant-fisher}, and in view of Remark~\ref{rem:identification}, we get the following equivalent characterization. 
\begin{cor}\label{cor:courant-fisher}
	For any $n \in \N^*$, it holds
	\[ 
	\mu_n^{(s)} = \min_{A \in \widetilde{M}_n^*} \max_{u \in A} \frac{\frac{c_{N,s}}{2}  [\widetilde{u}]_{H^s_{\Omega,0}}^2}{\|u\|_{L^2(\Omega)}^2},
	\]
	where $\widetilde{M}_n^*$ is the collection of $n$-dimensional subspaces of $H^s(\Omega) \cap L^2_0(\Omega)$.
\end{cor}

\section{Spectral stability for $s \to 1^-$}\label{sec:convergence}

Hereinafter, we let $s \in (0,1)$, $N \geq 2$,  $\Omega \subset \R^N$ be a bounded Lipschitz domain, and $\omega_k$ be the volume of a $k$-dimensional unit ball. 

The aim of this section is to prove the following stability result.
\begin{thm}\label{thm:stable}
	Let $n \in \mathbb{N}$. 
	Then
	\begin{equation}\label{eq:thm:stable:conver}
		\lim_{s \to 1^-}  \mu_n^{(s)} 
		= 
	\mu_n^{(1)}.
	\end{equation}
	Moreover, if $u_n^{(s)}$ is an eigenfunction corresponding to $\mu_n^{(s)}$ and such that $\|u_n^{(s)}\|_{L^2(\Omega)}=1$, then there exists an increasing sequence $\{s_k\}_{k \in \N^*} \subset (0,1)$ with $s_k \to 1^-$ such that
	$$
	\|u_{n}^{(s_k)} - u_n^{(1)}\|_{L^2(\Omega)} \to 0
	\quad \text{as}~ k \to +\infty,  
	$$
	where $u_n^{(1)}$ is an eigenfunction of the local Neumann Laplacian corresponding to $\mu_n^{(1)}$. 
\end{thm}

The proof of the theorem makes use of the notion of $\Gamma$-convergence (see for instance \cite{braides,dalmaso}), which we will briefly recall in relation to our setting. Let $(X,d)$ be a metric space, and let $\{\mathcal{F}_s\}_{s \in (0,1)}$ be a family of functionals $\mathcal{F}_s : X \to [0,+\infty]$. Similarly, let $\mathcal{F}_1 : X \to [0,+\infty]$ be a functional. We say that the functionals $\mathcal{F}_s$ \emph{$\Gamma$-converge} to $\mathcal{F}_1$ as $s \to 1^-$ if the following conditions are satisfied:
\begin{enumerate}[label={\rm(\roman*)}]
\item\label{liminf} {\it liminf inequality:} for every increasing sequence $\{s_k\}_{k \in \N^*} \subset (0,1)$ with $s_k \to 1^-$ as $k \to +\infty$,  for every $u \in X$, and for every sequence $\{u_{k}\}_{k \in \N^*} \subset X$ converging to $u$ in $X$ as $k \to +\infty$, it holds
\begin{equation}\label{liminf0}
\mathcal{F}_1(u) \leq \displaystyle \liminf_{k \to +\infty} \mathcal{F}_{s_k}(u_{k}).
\end{equation}
\item\label{limsup}	{\it limsup inequality:} for every increasing sequence $\{s_k\}_{k \in \N^*} \subset (0,1)$ with $s_k \to 1^-$ as $k \to +\infty$, and for every $u \in X$, there exists a sequence $\{u_{k}\}_{k \in \N^*} \subset X$ converging to $u$ in $X$ as $k \to +\infty$ such that
\begin{equation}\label{limsup0}
\mathcal{F}_1(u) \geq \displaystyle \limsup_{k \to +\infty} \mathcal{F}_{s_k}(u_{k}).
\end{equation}
\end{enumerate}
According to \cite[Proposition~1.44]{braides}, 
in order to establish \ref{liminf} and \ref{limsup}, and hence to prove the $\Gamma$-convergence as $s \to 1^-$, it is sufficient to justify \eqref{liminf0} and \eqref{limsup0} up to a subsequence of a given sequence of indices.

An equivalent way to define the $\Gamma$-convergence is the following. Denote
\[ \left(\Gamma\!-\!\liminf_{s \to 1^-} \mathcal{F}_s\right)(u) := \inf \left\{ \liminf_{k \to +\infty} \mathcal{F}_{s_k}(u_k)\,\bigg|\, \{s_k\}_{k \in \N^*} \subset (0,1),\,s_k \nearrow 1,\,u_k \to u \text{ as }k \to +\infty \right\},\]
\[ \left(\Gamma\!-\!\limsup_{s \to 1^-} \mathcal{F}_s\right)(u) := \inf \left\{ \limsup_{k \to +\infty} \mathcal{F}_{s_k}(u_k)\,\bigg|\, \{s_k\}_{k \in \N^*} \subset (0,1),\,s_k \nearrow 1,\,u_k \to u \text{ as }k \to +\infty \right\}.\]
Then the liminf inequality \ref{liminf} is equivalent to 
\begin{equation}\label{eq:gamma-inf}
	\mathcal{F}_1(u) \leq \left(\displaystyle \Gamma\!-\!\liminf_{s \to 1^-} \mathcal{F}_s\right)(u) \quad \text{for every }u \in X,
\end{equation}
and the limsup inequality \ref{limsup} is equivalent to
\begin{equation}\label{eq:gamma-sup}
	\mathcal{F}_1(u) \geq \left(\displaystyle \Gamma\!-\!\limsup_{s \to 1^-} \mathcal{F}_s\right)(u)\quad \text{for every }u \in X.
\end{equation}
It turns out that the functionals $\mathcal{F}_s$ $\Gamma$-converge to $\mathcal{F}_1$ as $s \to 1^-$ if and only if, for every $u \in X$,
\[ \mathcal{F}_1 (u) = \left(\Gamma\!-\!\liminf_{s \to 1^-} \mathcal{F}_s\right)(u) = \left( \Gamma\!-\!\limsup_{s \to 1^-} \mathcal{F}_s\right)(u),\] 
see \cite[Section~1.6]{braides}. 
In this case, we write
\begin{equation}\label{eq:gammaconv}
\mathcal{F}_1 (u) = \left(\Gamma\!-\!\lim_{s \to 1^-} \mathcal{F}_s\right) (u)
\quad \text{for every}~ u \in X.
\end{equation}

Prior to the proof of Theorem \ref{thm:stable}, we establish three auxiliary results. 

\begin{lem}
	\label{lem:limsup}
	Let $s_0 \in (0,1)$, and let $\{u_s\}_{s \in (s_0,1)}$ be a family of functions such that $u_s \in H^s(\Omega) \cap L^2_0(\Omega)$ for any $s \in (s_0,1)$. Let $\widetilde{u}_s \in H^s_{\Omega,0}$ be the $s$-minimal extension of $u_s$.	Assume that there exists $C>0$ such that
	\begin{equation}\label{eq:lem:comp:assumption:1}
		(1-s) \iint_{\mathbb{R}^{2N}\setminus (\Omega^c)^2}\frac{|\widetilde{u}_s(x)-\widetilde{u}_s(y)|^2}{|x-y|^{N+2s}}\,dx\,dy \leq C \quad \text{for every}~ s \in (s_0,1).
	\end{equation}
	Then there exists an increasing sequence $\{s_k\}_{k \in \N^*} \subset (s_0,1)$ with $s_k \to 1$ as $k \to +\infty$, and a function $u \in H^1(\Omega) \cap L^2_0(\Omega)$ such that $u_{s_k} \to u$ in $L^2(\Omega)$ as $k \to +\infty$. 
\end{lem}
\begin{proof}
	The Poincar\'e-Wirtinger inequality given by  \cite[Theorem~1.1]{ponce} implies that, for every $\varepsilon_0 > 0$, there exists a constant $C=C(\Omega,\varepsilon_0) > 0$ such that, for every $\varepsilon \in (0,\varepsilon_0)$,
	\begin{equation}\label{eq:PW1}
		\int_\Omega u^2 \,dx 
		\leq 
		C 
		\int_\Omega \int_\Omega \frac{|u(x)-u(y)|^2}{|x-y|^2} \rho_\varepsilon(|x-y|) \,dx\,dy \quad \text{for every }u \in L^2_0(\Omega),
	\end{equation}
where
	$$
	\rho_\varepsilon(t) 
	=
	\begin{cases}
		\displaystyle \frac{\varepsilon}{N \omega_N |t|^{N-\varepsilon}} &\text{for}~ |t|<1,\\
		0 &\text{for}~ |t| \geq 1.	
	\end{cases}
	$$ 
	Taking $\varepsilon_0 = 2(1-s_0)$ and $\varepsilon = 2(1-s)$, the inequality \eqref{eq:PW1} turns to
	\begin{equation}\label{eq:PW2}
		\int_\Omega u^2 \,dx
		\leq 
		C \, \frac{2(1-s)}{N \omega_N}
		\int_\Omega \int_\Omega \frac{|u(x)-u(y)|^2}{|x-y|^{N+2s}} \,dx\,dy
	\end{equation}
	for any $u \in L^2_0(\Omega)$ and $s \in (s_0, 1)$. 
	Applying \eqref{eq:PW2} to the family $\{u_s\}_{s \in (s_0,1)}$ given in the statement, we deduce that it is bounded in $L^2(\Omega)$. 
	Further applying \cite[Theorem~1.2]{ponce}, we conclude that $\{u_s\}_{s \in (s_0,1)}$ is relatively compact in $L^2(\Omega)$. 
	That is, there exists an increasing sequence $\{s_k\}_{k \in \mathbb{N}^*}$ such that $s_k \to 1^-$ and $\{u_{s_k}\}_{k \in \mathbb{N}^*}$ converges in $L^2(\Omega)$ to some $u \in L^2(\Omega)$ as $k \to +\infty$. 
	Clearly, we also have $\int_\Omega u \,dx = 0$, i.e., $u \in L^2_0(\Omega)$. 
	Moreover, \cite[Theorem~1.2]{ponce} yields $u \in H^1(\Omega)$. 	
	The proof is complete.
\end{proof}

\begin{lem}\label{lem:liminf}
	Let $u \in C^{\delta}(\mathbb{R}^N) \cap H^1(\Omega)$ for some $\delta \in (1/2,1)$. 
	Then $u \in H_{\Omega,0}^s$ for any $s \in (\delta,1)$ and 
	\[ 
	\lim_{s \to 1^-} (1-s)
	\iint_{\mathbb{R}^{2N}\setminus (\Omega^c)^2}\frac{|u(x)-u(y)|^2}{|x-y|^{N+2s}}\,dx\,dy
	=
	\frac{\omega_{N-1}}{2N} 
	\int_\Omega |\nabla u|^2 \,dx.
	\]
\end{lem}
\begin{proof}
	The result is established in \cite[Proposition~5.1]{dipierrorosotonvaldinoci} (see, more precisely, \cite[Eq.~(5.3)]{dipierrorosotonvaldinoci}) for  $u \in C_c^{2}(\mathbb{R}^N)$ and it is based on the results of \cite{BBM} and \cite{dinezzapalatuccivaldinoci}. 
	Let us provide a sketch of the proof which weakens the assumption on $u$ and fixes an imprecision in the derivation of \cite[Eq.~(5.5)]{dipierrorosotonvaldinoci}. 

	For any $u \in C^{\delta}(\mathbb{R}^N)$ with $\delta \in (1/2,1)$ and for any $s \in (\delta,1)$, we have
	$$
		\int_{\Omega} \int_{\Omega^c} \frac{|u(x)-u(y)|^2}{|x-y|^{N+2 s}} \,dx\,dy \leq \|u\|_{C^{\delta}(\mathbb{R}^N)}^2 \int_{\Omega} \int_{\Omega^c} \frac{1}{|x-y|^{N+2(s - \delta)}} \,dx\,dy
		=
		\frac{1}{2}\|u\|_{C^{\delta}(\mathbb{R}^N)}^2
		P_{2(s-\delta)}(\Omega),
	$$
where $P_{2(s-\delta)}(\Omega)$ is the fractional $2(s-\delta)$-perimeter of $\Omega$, see \cite[Section~4]{BLP}. 
We fix $\varepsilon \in (0,1-\delta)$, and note that $2(s-\delta) \in (\varepsilon,2(1-\delta)) \Subset (0,1)$ for any $s \in (\delta+\varepsilon,1)$. We now apply \cite[Corollary~4.4]{BLP} to obtain
$$
P_{2(s-\delta)}(\Omega)
\leq
\frac{2^{1-2(s-\delta)} N \omega_N}{(1-2(s-\delta)) \cdot 2(s-\delta)} |\partial \Omega|^{2(s-\delta)} |\Omega|^{1-2(s-\delta)}
< C < +\infty,
$$
where $C = C(\Omega,\delta,\varepsilon)>0$ does not depend on $s \in (\delta+\varepsilon,1)$. 
Therefore, using the decomposition \eqref{eq:decomp:1} and the embedding $H^1(\Omega) \subset H^s(\Omega)$ (see  \cite[Proposition~2.2]{dinezzapalatuccivaldinoci}), we get $u \in H_{\Omega,0}^s$ for any $s \in (\delta,1)$. Moreover, using the convergence result from \cite{BBM}, we arrive at
\begin{align}
	\lim _{s \to 1^-} & (1-s) \iint_{\mathbb{R}^{2 N} \backslash(\Omega^c)^2} \frac{|u(x)-u(y)|^2}{|x-y|^{N+2 s}} d x d y \\
	\label{eq:prop2:1}
	& =\lim _{s \rightarrow 1^-}(1-s) \int_{\Omega}\int_{\Omega} \frac{|u(x)-u(y)|^2}{|x-y|^{N+2 s}} d x d y
	=
	C(N) 
	\int_{\Omega}|\nabla u|^2 \,dx.
\end{align}
The explicit constant $C(N) = \frac{\omega_{N-1}}{2N}$ is derived as in the proof of \cite[Proposition~5.1]{dipierrorosotonvaldinoci} using the results of \cite{dinezzapalatuccivaldinoci}. 
\end{proof}

Recall that, for $s \in (0,1)$ and $u \in H^s(\Omega)$, we denote by $\widetilde{u} \in \widetilde{H}^s_{\Omega,0}$ the $s$-minimal extension of $u$ (see Proposition \ref{prop:minimalextension}). 
We define, for $s \in (0,1)$, the family of functionals $\mathcal{F}_s: L^2_0(\Omega) \to [0,+\infty]$ as
\[
\mathcal{F}_s(u)\;=\;
\begin{cases}
	(1-s)
	\displaystyle\iint_{\mathbb{R}^{2N}\setminus (\Omega^c)^2}\frac{|\widetilde{u}(x)-\widetilde{u}(y)|^2}{|x-y|^{N+2s}}\,dx\,dy\,,& u\in H^s(\Omega),\\[4mm]
	+\infty,& \text{otherwise},
\end{cases}
\]
and the functional $\mathcal{F}_1: L^2_0(\Omega) \to [0,+\infty]$ as
\[
\mathcal{F}_1(u)\;=\;
\begin{cases}
	\displaystyle \frac{\omega_{N-1}}{2N} \int_{\Omega}|\nabla u|^2\,dx\,,& u\in H^1(\Omega),\\[2mm]
	+\infty,& \text{otherwise.}
\end{cases}
\]

\begin{prop} \label{prop:gammaconv}
It holds	
\[ \mathcal{F}_1 (u) = \left(\Gamma\!-\!\lim_{s \to 1^-} \mathcal{F}_s\right) (u)\]
for every $u \in L^2_0(\Omega)$.
\end{prop}
\begin{proof}
In order to prove the $\Gamma$-convergence of $\mathcal{F}_s$ to $\mathcal{F}_1$ as $s \to 1^-$, we need to prove the liminf inequality \eqref{eq:gamma-inf}, and the limsup inequality \eqref{eq:gamma-sup}.
	
	\textbf{liminf inequality.} 
	Let $\{s_k\}_{k \in \N^*} \subset (0,1)$ be increasing and converging to $1$, let $u \in L^2_0(\Omega)$, and let $\{u_k\}_{k \in \N^*} \subset L^2_0(\Omega)$ be a sequence converging to $u$ in $L^2(\Omega)$.
	Observe that if
	$$
	\liminf_{k \to +\infty} \mathcal{F}_{s_k}(u_k)
	=
	+\infty,
	$$
	then there is nothing to prove. 
	Therefore, assume that
	\begin{equation}\label{eq:liminf:pr1}
	\liminf_{k \to +\infty} \mathcal{F}_{s_k}(u_k)
	< +\infty. 
	\end{equation}
	By the definition of $\mathcal{F}_s$, this	means that there exists a constant $C>0$ and a subsequence of $\{s_k\}_{k \in \N^*}$ (still denoted by the same index), such that $u_k \in H^{s_k}(\Omega)$ and
	$$
	(1-s_k) \iint_{\mathbb{R}^{2N}\setminus (\Omega^c)^2}\frac{|\widetilde{u}_k(x)-\widetilde{u}_k(y)|^2}{|x-y|^{N+2s_k}}\,dx\,dy \leq C
	$$
	for any $k \in \N^*$. Therefore, Lemma~\ref{lem:limsup} implies that $\{u_k\}_{k \in \N^*}$ converges in $L^2(\Omega)$ to some $u^* \in H^1(\Omega) \cap L^2_0(\Omega)$ as $k \to +\infty$, and hence $u^* \equiv u$. That is, $u \in H^1(\Omega) \cap L^2_0(\Omega)$. 
	
	Now we can argue verbatim as in the proof of \cite[Proposition 3.11]{brascoparinisquassina} with $p=2$ to obtain the chain of inequalities 
	\begin{align}
		\frac{\omega_{N-1}}{2N} \int_\Omega |\nabla u|^2 \,dx
		&\leq
		(1+\varepsilon) \liminf_{k \to +\infty} \sum_{i \in I} (1-s_k) [u_k]_{H^{s_k}(C_i)}^2\\
		&\leq (1+\varepsilon) \liminf_{k \to +\infty} (1-s_k) [u_k]_{H^{s_k}(\Omega)}^2
		\\
		&\leq
		(1+\varepsilon) \liminf_{k \to +\infty} (1-s_k)\iint_{\mathbb{R}^{2N}\setminus (\Omega^c)^2}\frac{|\widetilde{u}_k(x)-\widetilde{u}_k(y)|^2}{|x-y|^{N+2s_k}}\,dx\,dy,
	\end{align} 
	where $\varepsilon>0$ is arbitrary, and $\{C_i\}_{i \in I}$ is an appropriate family of disjoint cubes such that each $C_i \subset \Omega$. 
	
	This yields the $\Gamma\!-\!\liminf$ inequality \eqref{liminf0} for a subsequence of any sequence $\{s_k\}_{k \in \N^*}$ converging to $1^-$. 
	As we noted after \eqref{limsup0}, this fact is sufficient to get \eqref{eq:gamma-inf}.
	
	\medskip
	\textbf{limsup inequality.}
	Clearly, if $u \not\in H^1(\Omega)$, then there is nothing to prove. 
	Assume that $u \in H^1(\Omega)$. 
	Since $\Omega$ is Lipschitz, there exists a sequence $\{v_j\}_{j \in \N^*} \subset C_c^\infty(\mathbb{R}^N)$ 
	such that $v_j \to u$ in $H^1(\Omega)$. 
	Consider the functions
	$$
	u_j = \left(v_j - \frac{1}{|\Omega|} \int_\Omega v_j \,dx\right) \cdot f,
	\quad j \in \N^*,
	$$
	where $f \in C_c^\infty(\mathbb{R}^N)$ is such that $f = 1$ in a neighborhood of $\Omega$. 
	We deduce that $u_j \in C_c^\infty(\mathbb{R}^N) \cap L^2_0(\Omega)$ for any $j \in \N^*$, and 
	$u_j \to u$ in $H^1(\Omega)$ as $j \to +\infty$. 
	
	Let  $\{s_k\}_{k \in \N^*} \subset (0,1)$ be an increasing sequence converging to $1$. 
	For each fixed $u_j$, Lemma~\ref{lem:liminf} yields
	\[ 
	\lim_{k \to +\infty} (1-s_k)
	\iint_{\mathbb{R}^{2N}\setminus (\Omega^c)^2}\frac{|u_j(x)-u_j(y)|^2}{|x-y|^{N+2s_k}}\,dx\,dy
	=
	\frac{\omega_{N-1}}{2N} \int_\Omega |\nabla u_j|^2 \,dx.
	\]
	For each $k \in \N^*$, we denote $\widetilde{u}_{j,k}$ be the $s_k$-minimal extension of $u_j$. 
	Using Proposition~\ref{prop:minimalextension}, we obtain
	\begin{align*} 
	\limsup_{k \to +\infty} \mathcal{F}_{s_k}(u_j)  & = 	
	\limsup_{k \to +\infty} \, (1-s_k)
	\iint_{\mathbb{R}^{2N}\setminus (\Omega^c)^2}\frac{|\widetilde{u}_{j,k}(x)-\widetilde{u}_{j,k}(y)|^2}{|x-y|^{N+2s_k}}\,dx\,dy \\ & \leq
	\limsup_{k \to +\infty} \, (1-s_k)
	\iint_{\mathbb{R}^{2N}\setminus (\Omega^c)^2}\frac{|u_j(x)-u_j(y)|^2}{|x-y|^{N+2s_k}}\,dx\,dy
	\\ & =
	\frac{\omega_{N-1}}{2N} \int_\Omega |\nabla u_j|^2 \,dx = \mathcal{F}_1(u_j).
	\end{align*}
	Thus, we have shown that the $\Gamma\!-\!\limsup$ inequality \eqref{limsup0} holds true for every $u_j$.	Noting now that
	$u_j \to u$ in $H^1(\Omega)$ as $j \to +\infty$, and since the $\Gamma$-upper limit is lower semicontinuous (see, e.g., \cite[Proposition~1.28]{braides}), we arrive at
	\[ \left(\Gamma\!-\!\limsup_{s \to 1^-} \mathcal{F}_{s}\right)(u) 
	\leq 
	\liminf_{j \to +\infty} \left(\Gamma\!-\!\limsup_{s \to 1^-} \mathcal{F}_{s}\right)(u_j) \leq \liminf_{j \to +\infty} \mathcal{F}_1(u_j) = \mathcal{F}_1(u). \]
	That is, the $\Gamma\!-\!\limsup$ inequality \eqref{eq:gamma-sup} is satisfied.
\end{proof}

Now we are ready to prove Theorem~\ref{thm:stable}.

\begin{proof}[Proof of Theorem \ref{thm:stable}]
	By Proposition \ref{prop:gammaconv}, we know that the functionals $\mathcal{F}_s$ $\Gamma$-converges to $\mathcal{F}_1$ as $s \to 1^-$ in $L^2_0(\Omega)$. 
	The convergence of critical values of $\mathcal{F}_s$ to those of $\mathcal{F}_1$
	as $s \to 1^-$ then follows as in \cite[Section~4.1]{brascoparinisquassina}, by exploiting \cite[Corollary~4.4]{degiovannimarzocchi}\footnote{The work \cite{degiovannimarzocchi} actually deals with the spectrum of nonlinear operators defined by means of the \emph{Krasnosel'skii genus}. However, in the linear case $p=2$, this notion coincides with the spectrum obtained by the classical Courant-Fisher theory, see, e.g., \cite[Appendix~A]{brascoparinisquassina} for the case of the Dirichlet boundary conditions.}. 
	Moreover, thanks to Remark~\ref{rem:identification} and  Corollary~\ref{cor:courant-fisher}, \cite[Corollary~3.3]{degiovannimarzocchi} implies that these critical values coincide with eigenvalues up to appropriate multipliers.
	Recalling that for any eigenfunction $u_n^{(s)} \in \widetilde{H}_{\Omega,0}^s \cap L^2_0(\Omega)$ corresponding to $\mu_n^{(s)}$ and satisfying $\|u_n^{(s)}\|_{L^2(\Omega)}=1$ it holds 
	\begin{equation}\label{eq:cNs-limit1}
	\mu_n^{(s)} = \frac{c_{N,s}}{2}[u_n^{(s)}]_{H_{\Omega,0}^s}^2 = \frac{c_{N,s}}{2(1-s)} \mathcal{F}_s(u_n^{(s)})
	\end{equation}
	(where the second equality is due to Remark~\ref{rem:identification})
	and that
	\begin{equation}\label{eq:cNs-limit2}
		\frac{c_{N,s}}{2} \sim \frac{2N}{\omega_{N-1}} (1-s)
		\quad \text{as}~ s \to 1^-
	\end{equation}
	(see, e.g.,  \cite[Corollary~4.2]{dinezzapalatuccivaldinoci}), we deduce that
	\begin{equation}\label{eq:cNs-limit3} 
	\lim_{s \to 1^-} \mu_{n}^{(s)} = \mu_n^{(1)}
	\quad \text{for every}~ n \in \N.
	\end{equation}
	
	Let us prove the convergence of eigenfunctions in $L^2(\Omega)$. We follow the approach from \cite[Section~4.2]{brascoparinisquassina}. 
	In view of \eqref{eq:cNs-limit1}, \eqref{eq:cNs-limit2}, \eqref{eq:cNs-limit3}, we get the existence of $s_0 \in (0,1)$ such that  \[ \sup_{s \in (s_0,1)} (1-s)[u_n^{(s)}]_{H_{\Omega,0}^s}^2 < 
	+\infty.\]
	Thanks to Lemma~\ref{lem:limsup}, there exists $\{s_k\}_{k \in \N^*} \subset (s_0,1)$ converging to $1$ such that $\{u_{n}^{(s_k)}\}_{k \in \N^*}$ converges in $L^2(\Omega)$ to some $u_n \in H^1(\Omega) \cap L^2_0(\Omega)$ as $k \to +\infty$. 
	In particular, we also have $\|u_n\|_{L^2(\Omega)}=1$. 
	
	Let us show that $u_n$ is an eigenfunction of the local Neumann Laplacian in $\Omega$. 
	Observe that each $u_n^{(s_k)}$ is a unique minimizer of the strictly convex functional $\mathcal{G}_{s_k}:L^2_0(\Omega) \to \R$ defined as
	$$
	\mathcal{G}_{s_k}(v)
	:=
	\frac{1}{2}\mathcal{F}_{s_k}(v) - \frac{2 (1-s_k) \mu_n^{(s)}}{c_{N,s}} \int_\Omega u_{n}^{(s_k)} v \,dx. 
	$$
	Due to \eqref{eq:cNs-limit2}, \eqref{eq:cNs-limit3}, and the strong convergence of $\{u_n^{(s_k)}\}_{k \in \N^*}$ to $u_n$ in $L^2(\Omega)$, it is not hard to observe by the triangle inequality that the functional
	$$
	v \in L^2_0(\Omega) \mapsto -\frac{2 (1-s_k) \mu_n^{(s)}}{c_{N,s}} \int_\Omega u_n^{(s_k)} v \,dx 
	$$
	continuously converges (in $L^2(\Omega))$ to the functional
	$$
	v \in L^2_0(\Omega) \mapsto - \,\frac{\omega_{N-1}}{2N}\mu_n^{(1)} \int_\Omega u_n v \,dx,
	$$
	see \cite[Definition~4.7]{dalmaso}. 
	Therefore, \cite[Proposition~6.20]{dalmaso} implies that $\mathcal{G}_{s_k}$ $\Gamma$-converges in $L^2(\Omega)$ to the functional $\mathcal{G}_{1}:L^2_0(\Omega) \to \R$ defined as
	$$
	\mathcal{G}_1(v) 
	:=
	\frac{1}{2}	\mathcal{F}_{1}(v) -  \,\frac{\omega_{N-1}}{2N} \mu_n^{(1)}\int_\Omega u_n v \,dx. 
	$$	
	Since $u_n^{(s_k)} \to u_n$ in $L^2_0(\Omega)$ as $k \to +\infty$, $u_n$ must be a minimizer of $\mathcal{G}_1$. 
	Indeed, for every $v \in H^1(\Omega) \cap L^2_0(\Omega)$ there exists a recovery sequence $\{v_k\}_{k \in \N^*} \subset L^2_0(\Omega)$ with $v_k \to v$ in $L^2_0(\Omega)$ as $k \to +\infty$ satisfying the limsup inequality. This implies
	\[ \mathcal{G}_1(u_n) \leq \liminf_{k \to +\infty} \mathcal{G}_{s_k}\left(u_n^{(s_k)}\right) \leq \limsup_{k \to +\infty} \mathcal{G}_{s_k}\left(u_n^{(s_k)}\right) \leq  \limsup_{k \to +\infty} \mathcal{G}_{s_k}(v_k) \leq \mathcal{G}_1(v).\]
	By the strict convexity of $\mathcal{G}_1$, $u_n$ is its unique minimizer. Consequently, $u_n$ is an eigenfunction of the local Neumann Laplacian in $\Omega$ associated to $\mu_n^{(1)}$.
\end{proof}

\section{Symmetry results}\label{sec:symmetry}

The purpose of this section is to prove Theorems~\ref{thm:main1} and \ref{thm:main2}. 
Prior to that, we provide several auxiliary definitions and results related to symmetry properties of eigenfunctions. 

Hereinafter, we let $s \in (0,1)$, $N \geq 2$, and $B \subset \R^N$ be a ball of radius $R>0$ centered at the origin. 

\begin{definition}
	Let $e \in \mathbb{S}^{N-1}$. Let \[\Sigma_e^+:= \{x \in \R^N\,|\,x \cdot e \geq 0\}, \qquad \Sigma_e^-:= \{x \in \R^N\,|\,x \cdot e \leq 0\}, \qquad H_e := \{x \in \R^N\,|\,x \cdot e = 0\}.\] For $x \in \R^N$, let $\sigma_e(x) \in \R^N$ be its symmetric point with respect to $H_e$. For $u : \R^N \to \R$, the \emph{polarization} with respect to $H_e$ is the function $P_e u :\R^N \to \R$ defined as
	\[ (P_a u)(x) := \left\{ \begin{array}{c l} \min\{u(x),u(\sigma_e(x))\} & \text{if }x \in \Sigma_e^+, \\ \max\{u(x),u(\sigma_e(x))\} & \text{if }x \in \Sigma_e^-.
	\end{array}\right.
	\]
\end{definition}
It is well-known that 
\begin{gather}
	\label{eq:polarization-normweak} 
	\|P_e u\|_{L^2(\mathbb{R}^N)}
	=
	\|u\|_{L^2(\mathbb{R}^N)}
	\quad
	\text{for any}~ u \in L^2(\mathbb{R}^N),\\
	\label{eq:polarization-normweak2} 
	\int_{\mathbb{R}^N} P_e u \,dx
	=
	\int_{\mathbb{R}^N} u \,dx
	\quad
	\text{for any}~ u \in L^1(\mathbb{R}^N), 
\end{gather}
see, e.g., \cite[Eq.~(3.7)]{brocksol} applied to $u^+$ and $u^-$, where $u^+ = \max\{u,0\}$ and $u^- = \min\{u,0\}$. 
In particular, we have
\begin{gather}
	\label{eq:polarization-normweak3} 
	\|P_e u\|_{L^2(B)}
	=
	\|P_e (u \, \chi_B)\|_{L^2(\mathbb{R}^N)}
	=
	\|u \, \chi_B\|_{L^2(\mathbb{R}^N)}
	=
	\|u\|_{L^2(B)}
	\quad 
	\text{for any}~ u \in L^2(\mathbb{R}^N),\\
	\label{eq:polarization-normweak4}
	\int_{B} P_e u \,dx
	=
	\int_{\mathbb{R}^N} P_e (u \, \chi_B) \,dx
	=
	\int_{\mathbb{R}^N} u \, \chi_B \,dx
	=
	\int_{B} u \,dx
	\quad
	\text{for any}~ u \in L^1(\mathbb{R}^N).
\end{gather}

\begin{definition}\label{def:foliated}
	A function $u$ is \emph{foliated Schwarz symmetric} with respect to $p \in \mathbb{S}^{N-1}$ if:
	\begin{enumerate}[label={\rm(\roman*)}]
		\item $u$ is axially symmetric with respect to $\R p$, so that
		\[ u(x) = \hat{u}\left( |x|, \arccos\left( \frac{x}{|x|} \cdot p\right)\right)\]
		for some function $\hat{u}:[0,+\infty) \times [0,\pi] \to \R$.
		\item $\hat{u}$ is nonincreasing with respect to the second variable.
	\end{enumerate}
\end{definition}

We now prove that the operation of polarizing a function from $H^s_{B,0}$ decreases its $H^s_{B,0}$-seminorm, analogously to the same property of the Gagliardo seminorm, cf.\ \cite{BE}. 
Moreover, we provide a characterization of the equality cases. 

\begin{prop}\label{prop:polarization}
	Let $e \in \mathbb{S}^{N-1}$. 
	Let $u \in H^s_{B,0}$, and let $P_e u$ be its polarization with respect to $H_e$.
Then $P_e u \in H^s_{B,0}$ and	
\begin{equation}\label{eq:polarization-norm} 
		\iint_{\R^{2N} \setminus (B^c)^2} \frac{|(P_e u)(x)-(P_e u)(y)|^2}{|x-y|^{N+2s}}\,dx\,dy \leq \iint_{\R^{2N} \setminus (B^c)^2} \frac{|u(x)-u(y)|^2}{|x-y|^{N+2s}}\,dx\,dy.
	\end{equation}
Moreover, equality holds in \eqref{eq:polarization-norm} if and only if either $P_e u = u$ or $P_e u = u \circ \sigma_e$ in $\mathbb{R}^N$.
\end{prop}
\begin{proof}
		The proof of \eqref{eq:polarization-norm} goes as in \cite[Lemma 2.3]{benediktbobkovdharagirg}, by using the decomposition 
		\begin{equation}\label{eq:decomp} 
		\R^{2N} \setminus (B^c)^2 = (B \times B) \cup (B^c \times B) \cup (B \times B^c)
		\end{equation}
		and showing that the integral on each subset decreases under polarization. 
		The equality cases can be characterized as in the proof of \cite[Proposition~2.1]{bobkol}. 
		The fact that $P_e u \in H^s_{B,0}$ will then follow from \eqref{eq:polarization-normweak3} and \eqref{eq:polarization-norm}. 
		We provide details for clarity.	
		
		Let us denote $v = P_e u$. 
		We have
		$$
		\iint_{\R^{2N} \setminus (B^c)^2} \frac{|v(x)-v(y)|^2}{|x-y|^{N+2s}}\,dx\,dy
		=
		\iint_{\R^{2N}} \frac{|v(x)-v(y)|^2}{|x-y|^{N+2s}} \,\chi(x,y)\,dx\,dy,
		$$
		where $\chi$ is the indicator function of the set $\R^{2N} \setminus (B^c)^2$. 
		The same equality holds for $u$. 
		Decomposing
		$$
	\mathbb{R}^{2N} = (\Sigma_e^+ \times \Sigma_e^+) \cup (\Sigma_e^+ \times \Sigma_e^-)  \cup (\Sigma_e^- \times \Sigma_e^+) \cup (\Sigma_e^- \times \Sigma_e^-)
		$$
		and noting that the set $\R^{2N} \setminus (B^c)^2$ is symmetric with respect to the reflection $\sigma_e$, 		we have
		\begin{gather} 			\label{eq:prop:decomp1}		
			\iint_{\R^{2N} \setminus (B^c)^2} \frac{|v(x)-v(y)|^2}{|x-y|^{N+2s}}
			\,dx\,dy
			 = 
			\iint_{(\Sigma_e^+)^2}
			\Bigg[ \frac{|v(x)-v(y)|^2}{|x-y|^{N+sp}} 
			+ \frac{|v(\sigma_e(x))-v(y)|^2}{|\sigma_e(x)-y|^{N+sp}}				
			\\ +
			\frac{|v(x)-v(\sigma_e(y))|^2}{|x-\sigma_e(y)|^{N+sp}}+ \frac{|v(\sigma_e(x))-v(\sigma_e(y))|^2}{|\sigma_e(x)-\sigma_e(y)|^{N+sp}}
			\Bigg] \chi(x,y)\, dx\,dy,
		\end{gather}
		and an analogous representation holds for $u$.
		Thus, in order to prove \eqref{eq:polarization-norm}, it is sufficient to establish the inequality
		\begin{multline}
			\frac{|v(x)-v(y)|^2}{|x-y|^{N+sp}} + \frac{|v(\sigma_e(x))-v(y)|^2}{|\sigma_e(x)-y|^{N+sp}} +  \frac{|v(x)-v(\sigma_e(y))|^2}{|x-\sigma_e(y)|^{N+sp}} + \frac{|v(\sigma_e(x))-v(\sigma_e(y))|^2}{|\sigma_e(x)-\sigma_e(y)|^{N+sp}}
			\\
			\label{eq:8}
			\leqslant
			\frac{|u(x)-u(y)|^2}{|x-y|^{N+sp}} + \frac{|u(\sigma_e(x))-u(y)|^2}{|\sigma_e(x)-y|^{N+sp}} +  \frac{|u(x)-u(\sigma_e(y))|^2}{|x-\sigma_e(y)|^{N+sp}} + \frac{|u(\sigma_e(x))-u(\sigma_e(y))|^2}{|\sigma_e(x)-\sigma_e(y)|^{N+sp}}
		\end{multline}
		for almost all points $(x, y) \in (\Sigma_e^+)^2 \cap (\R^{2N} \setminus (B^c)^2)$, and characterize equality cases. 
		
		It is not hard to see that for every $x, y \in \Sigma_e^+$ we have
		\begin{equation}\label{eq:7}
			\frac{1}{|x-y|^{N+sp}} 
			= 
			\frac{1}{|\sigma_e(x)-\sigma_e(y)|^{N+sp}}
			\geq 
			\frac{1}{|\sigma_e(x)-y|^{N+sp}} 
			= 
			\frac{1}{|x-\sigma_e(y)|^{N+sp}},
		\end{equation}
		where the equality occurs if and only if $x \in H_e$ or $y \in H_e$. 
		As a consequence of \eqref{eq:7}, we also have
		\begin{equation}\label{eq:7x}
			\frac{1}{|x-y|^{N+sp}} 
			-
			\frac{1}{|x-\sigma_e(y)|^{N+sp}}
			= 
			\frac{1}{|\sigma_e(x)-\sigma_e(y)|^{N+sp}} 
			-
			\frac{1}{|\sigma_e(x)-y|^{N+sp}} \geq 0,
		\end{equation}
	with the same characterization of equality cases. 
		
		Let us denote
		\begin{equation}\label{eq:sigma-tilde}
		\widetilde{\Sigma} = (\Sigma_e^+)^2 \cap (\R^{2N} \setminus (B^c)^2)
		\end{equation}
		and decompose $\widetilde{\Sigma}$ as the union of the following (nondisjoint) subsets:
		\begin{align}
			A_{++} &= \{(x,y) \in \widetilde{\Sigma}:~ 
			u(\sigma_e(x)) \geqslant u(x) 
			~\text{and}~
			u(\sigma_e(y)) \geqslant u(y) 
			\},\\
			A_{--} &= \{(x,y) \in \widetilde{\Sigma}:~ 
			u(\sigma_e(x)) \leqslant u(x) 
			~\text{and}~
			u(\sigma_e(y)) \leqslant u(y) 
			\},\\
			\label{eq:Apm}
			A_{+-} &= \{(x,y) \in \widetilde{\Sigma}:~ 
			u(\sigma_e(x)) > u(x) 
			~\text{and}~
			u(\sigma_e(y)) < u(y) 
			\},\\
			\label{eq:Amp}
			A_{-+} &= \{(x,y) \in \widetilde{\Sigma}:~ 
			u(\sigma_e(x)) < u(x) 
			~\text{and}~
			u(\sigma_e(y)) > u(y) 
			\}.
		\end{align}
		We investigate the inequality \eqref{eq:8} in each subset separately.
		
		$\bullet$
		Take any $(x,y) \in A_{++}$.
		By the definition of polarization, we have 
		\begin{equation}\label{eq:v=u1}
			v(x) = u(x), \quad v(\sigma_e(x)) = u(\sigma_e(x)),
			\quad \text{and} \quad 
			v(y) = u(y), \quad v(\sigma_e(y)) = u(\sigma_e(y)),
		\end{equation}
		and hence the inequality \eqref{eq:8} becomes equality for $(x,y) \in A_{++}$.
		
		$\bullet$
		Take any $(x,y) \in A_{--}$. 
		By the definition of polarization, we have
		\begin{equation}\label{eq:v=u2}
			v(x) = u(\sigma_e(x)), \quad v(\sigma_e(x)) = u(x),
			\quad \text{and} \quad 
			v(y) = u(\sigma_e(y)), \quad v(\sigma_e(y)) = u(y).
		\end{equation}
		Using the equalities from \eqref{eq:7}, we rewrite the left-hand side of \eqref{eq:8} as 
		\begin{equation}\label{lem:prof:x}
			\frac{|u(\sigma_e(x))-u(\sigma_e(y))|^2}{|\sigma_e(x)-\sigma_e(y)|^{N+sp}} +  \frac{|u(x)-u(\sigma_e(y))|^2}{|x-\sigma_e(y)|^{N+sp}} + 
			\frac{|u(\sigma_e(x))-u(y)|^2}{|\sigma_e(x)-y|^{N+sp}} +
			\frac{|u(x)-u(y)|^2}{|x-y|^{N+sp}}.
		\end{equation}
		This expression coincides with the right-hand side of \eqref{eq:8}. That is, \eqref{eq:8} holds as equality for $(x,y) \in A_{--}$.
		
		$\bullet$
		Take any $(x,y) \in A_{+-}$.
		By the definition of polarization, we have
		$$
		v(x) = u(x), \quad v(\sigma_e(x)) = u(\sigma_e(x)),
		\quad \text{and} \quad 
		v(y) = u(\sigma_e(y)), \quad v(\sigma_e(y)) = u(y),
		$$
		and hence we rewrite \eqref{eq:8} as
		\begin{multline}
			\frac{|u(x)-u(\sigma_e(y))|^2}{|x-y|^{N+sp}} + \frac{|u(\sigma_e(x))-u(\sigma_e(y))|^2}{|\sigma_e(x)-y|^{N+sp}} +  \frac{|u(x)-u(y)|^2}{|x-\sigma_e(y)|^{N+sp}} + \frac{|u(\sigma_e(x))-u(y)|^2}{|\sigma_e(x)-\sigma_e(y)|^{N+sp}}\\
			\label{eq:FFFF}
			\leqslant
			\frac{|u(x)-u(y)|^2}{|x-y|^{N+sp}} + \frac{|u(\sigma_e(x))-u(y)|^2}{|\sigma_e(x)-y|^{N+sp}} +  \frac{|u(x)-u(\sigma_e(y))|^2}{|x-\sigma_e(y)|^{N+sp}} + \frac{|u(\sigma_e(x))-u(\sigma_e(y))|^2}{|\sigma_e(x)-\sigma_e(y)|^{N+sp}}.
		\end{multline}
		Rearranging, we have
		\begin{align}
			&|u(x)-u(y)|^2
			\left(
			\frac{1}{|x-y|^{N+sp}}
			-
			\frac{1}{|{x}-\sigma_e(y)|^{N+sp}}
			\right)
			\\
			&-
			|u(\sigma_e(x))-u(y)|^2
			\left(
			\frac{1}{|\sigma_e(x)-\sigma_e(y)|^{N+sp}}
			-
			\frac{1}{|\sigma_e(x)-y|^{N+sp}}
			\right)
			\\
			&-
			|u(x)-u(\sigma_e(y))|^2
			\left(
			\frac{1}{|x-y|^{N+sp}}
			-
			\frac{1}{|{x}-\sigma_e(y)|^{N+sp}}
			\right)
			\\
			&+
			|u(\sigma_e(x))-u(\sigma_e(y))|^2
			\left(
			\frac{1}{|\sigma_e(x)-\sigma_e(y)|^{N+sp}}
			-
			\frac{1}{|\sigma_e(x)-y|^{N+sp}}
			\right) \geqslant 0.
			\label{eq:FFFF2}
		\end{align}
		Applying the equality in \eqref{eq:7x}, we rewrite \eqref{eq:FFFF2} as
		\begin{multline}
			\left[
			|u(x)-u(y)|^2 - |u(\sigma_e(x))-u(y)|^2 - |u(x)-u(\sigma_e(y))|^2 + |u(\sigma_e(x))-u(\sigma_e(y))|^2
			\right] \\
			\label{four_point_formula0}
			\times
			\left(
			\frac{1}{|x-y|^{N+sp}}
			-
			\frac{1}{|{x}-\sigma_e(y)|^{N+sp}}
			\right)
			\geqslant 0.
		\end{multline}
		Thanks to the inequality in \eqref{eq:7x} and the characterization of equality cases therein, we conclude that \eqref{four_point_formula0} is equivalent to the four-point inequality
		\begin{equation}\label{four_point_formula}
			F(x,y)
			:=
			|u(x)-u(y)|^2 - |u(\sigma_e(x))-u(y)|^2 - |u(x)-u(\sigma_e(y))|^2 + |u(\sigma_e(x))-u(\sigma_e(y))|^2
			\geqslant 0.
		\end{equation}
		Opening the brackets and rearranging, we observe that
		$$
		F(x,y) = 
		2 (u(\sigma_e(x)) - u(x)) (u(y) - u(\sigma_e(y)).
		$$
		Recalling that $(x,y) \in A_{+-}$, we indeed get $F(x,y)>0$. 
		This means that the inequality \eqref{eq:8} is strict in $A_{+-}$. 
		Arguing in much the same way, we obtain the strict inequality
		\eqref{eq:8} in $A_{-+}$. 

		Combining all four cases, we conclude that \eqref{eq:8} is satisfied in $\widetilde{\Sigma}$, which proves \eqref{eq:polarization-norm}.

		Let us now suppose that equality takes place in \eqref{eq:polarization-norm}. By inspection of the equality case for $A_{+-}$ and $A_{-+}$, it follows that 
		$|A_{+-}|=0$ and $|A_{-+}|=0$. 
		Therefore, we have the following cases:
		\begin{itemize}
			\item $u(\sigma_e(x))=u(x)$ for a.e.\ $x \in \mathbb{R}^N$. This implies $P_e u = u  = u \circ \sigma_e$.
			\item The set $\{x \in \Sigma_e^+\,|\,u(\sigma_e(x))>u(x)\}$ is not negligible. Since $|A_{+-}|=0$, it follows that $u(\sigma_e(x)) \geq u(x)$ for a.e.\ $x \in \Sigma_e^+$, which implies $P_e u = u$.
			\item The set $\{x \in \Sigma_e^+\,|\,u(\sigma_e(x))<u(x)\}$ is not negligible. Since $|A_{-+}|=0$, it follows that $u(\sigma_e(x)) \leq u(x)$ for a.e.\ $x \in \Sigma_e^+$, which implies $P_e u = u \circ \sigma_e$.
		\end{itemize}
		Thus, we always have either $P_e u = u$ or $P_e u = u \circ \sigma_e$. 
\end{proof}

\begin{cor}\label{cor:pol}
	Let $u$ be a first nontrivial eigenfunction. Then, for every $e \in \mathbb{S}^{N-1}$, either $P_e u = u$ or $P_e u = u \circ \sigma_e$ in $\mathbb{R}^N$, and hence $u$ is foliated Schwarz symmetric with respect to some $p \in \mathbb{S}^{N-1}$.
\end{cor}
\begin{proof}
	Let $\|u\|_{L^2(B)}=1$, and let $v \in H^s_{B,0}$ be defined as $v:=P_e u$ for some fixed $e \in \mathbb{S}^{N-1}$. 
	In view of \eqref{eq:polarization-normweak3} and \eqref{eq:polarization-normweak4}, we have $\|v\|_{L^2(B)}=1$ and $\int_{B} v \,dx = 0$, respectively. 
	Using Proposition~\ref{prop:polarization} and the variational characterization \eqref{eq:mu1} of $\mu_1^{(s)}$, we conclude that $v$ is a minimizer of $\mu_1^{(s)}$, and hence $v$ is also a first nontrivial eigenfunction. 
	Proposition~\ref{prop:polarization} further yields $v = u$ or $v=u \circ \sigma_e$. 
	Recalling from Section~\ref{sec:prelim} that $u \in C(\mathbb{R}^N)$, and applying \cite[Proposition~3.2]{SW} to $u$ in any ball centered at the origin, we deduce that $u$ is foliated Schwarz symmetric with respect to some $p \in \mathbb{S}^{N-1}$. 
\end{proof}

\begin{prop}\label{prop:twonodaldomains}
	Let $u$ be a first nontrivial eigenfunction which is nonradial. 
	Let $p \in \mathbb{S}^{N-1}$ be such that $u$ is foliated Schwarz symmetric with respect to $p$.
	Then $u - u \circ \sigma_p$ is a first nontrivial eigenfunction which is antisymmetric with respect to $H_p$, foliated Schwarz symmetric with respect to $p$, and has exactly two nodal domains in $B$.
\end{prop}
\begin{proof}
	By Corollary \ref{cor:pol}, there exists $p \in \mathbb{S}^{N-1}$ such that $u$ is foliated Schwarz symmetric with respect to $p$. Since $u$ is nonradial, we deduce that $u \not\equiv u \circ \sigma_p$. Set $v := u - u \circ \sigma_p$, so that $v \not \equiv 0$. 
	Therefore, $v$ is a first nontrivial eigenfunction which is antisymmetric with respect to $H_p$.
	By construction, $v$ is foliated Schwarz symmetric with respect to $p$,  $v \geq 0$ in $\Sigma_p^+$ and $v \leq 0$ in $\Sigma_p^-$.
	It is not hard to see that $v$, being a solution of \eqref{eq:weak}, satisfies the assumptions of the strong maximum principle for antisymmetric supersolutions \cite[Proposition~3.6]{jarohswethantisym} in any open subset $U \Subset B \cap \Sigma_p^+$, which implies that $v>0$ in any compact subset $K$ of $U$. 
	Consequently, $v$ has exactly two nodal domains in $B$. 
\end{proof}

\begin{prop}\label{prop:unique-antisym}
	For any $e \in \mathbb{S}^{N-1}$, there exists at most one (modulo scaling) first nontrivial eigenfunction which is antisymmetric with respect to $H_e$. 
\end{prop}
\begin{proof}
	Let $e \in \mathbb{S}^{N-1}$, and let $u, v \in \widetilde{H}^s_{B,0}$ be two first nontrivial eigenfunctions which are antisymmetric with respect to $H_e$ and normalized as \[\|u\|_{L^2(B)}=\|v\|_{L^2(B)}=1.\]
	Corollary~\ref{cor:pol} implies that either $P_e u = u$ or $P_e u  = u\circ \sigma_e$, and hence, without loss of generality, $u \geq 0$ in $\Sigma_e^+$. 
	In the same way, we can also assume that $v \geq 0$ in $\Sigma_e^+$. 
	Suppose that $u \not \equiv v$ and define $w:=u-v$. 
	In particular, $w$ is also a first nontrivial eigenfunction which is antisymmetric with respect to $H_e$. 
	Again in view of Corollary~\ref{cor:pol} we have either $P_e w = w$ or $P_e w  = w\circ \sigma_e$, and hence, without loss of generality, $w \geq 0$ in $\Sigma_e^+$. 
	As in the proof of Proposition~\ref{prop:twonodaldomains}, the strong maximum principle  \cite[Proposition~3.6]{jarohswethantisym} gives either $w=0$ or $w>0$ in $B \cap \Sigma_e^+$. 
	If $w>0$ in $B \cap \Sigma_e^+$, then $u>v \geq 0$ in $B \cap \Sigma_e^+$, which leads to the following contradiction:
	\[
	1
	=
	\|u\|_{L^2(B)}^2 
	=
	2 \|u\|_{L^2(B\cap \Sigma_e^+)}^2
	> 
	2\|v\|_{L^2(B\cap \Sigma_e^+)}^2
	=
	\|v\|_{L^2(B)}^2
	=
	1.
	\]
	Thus, we have $w=0$ in $B \cap \Sigma_e^+$, and hence $u=v$ in $B$. 
	However, since $u, v \in \widetilde{H}^s_{B,0}$, Proposition~\ref{prop:minimalextension} yields $u \equiv v$, which contradicts our initial assumption. 
\end{proof}

\begin{prop}\label{prop:radial}
	Let $u \in H^s_{B,0}$ be a first nontrivial eigenfunction which is not antisymmetric with respect to $H_e$ for any $e \in \mathbb{S}^{N-1}$. 
	Let $p \in \mathbb{S}^{N-1}$ be such that $u$ is foliated Schwarz symmetric with respect to $p$.
	Then $u + u \circ \sigma_p$ is a first nontrivial eigenfunction which is radially symmetric. 
\end{prop}
\begin{proof}
	By the assumption, we have $u \not\equiv -u \circ \sigma_p$.  
	Consider $v := u + u \circ \sigma_p$, so that $v \not \equiv 0$ and $v$ is a first nontrivial eigenfunction. 
	Suppose, by contradiction, that $v$ is nonradial. 
	Since $u$ and $u \circ \sigma_p$ are axially symmetric with respect to $\R p$ (see Definition~\ref{def:foliated}), so is $v$: 
	\[ 
	v(x) = \hat{v}\left(r, \theta\right),
	\quad \text{where}~
	r = |x|,
	\quad
	\theta = \arccos\left( \frac{x}{|x|} \cdot p\right),
	\]
	for some $\hat{v}:[0,+\infty) \times [0,\pi] \to \R$, and, moreover, $\hat{v}(r,\theta) = \hat{v}(r,\pi - \theta)$ by construction. 
	By Corollary~\ref{cor:pol}, $v$ is foliated Schwarz symmetric with respect to some $p' \in \mathbb{S}^{N-1}$. 
	Recalling that $v$ is nonradial, we conclude that $p'$ has to be collinear to $p$. Since $\hat{v}$ is nonincreasing with respect to $\theta$ and $\hat{v}(r,\theta) = \hat{v}(r,\pi - \theta)$, we deduce that $\hat{v}$ does not depend on $\theta$, a contradiction to the nonradiality assumption. Therefore, $v$ is radially symmetric.
\end{proof}

We now have all the necessary ingredients to prove Theorem \ref{thm:main1} on the structure of the first nontrivial eigenspace in the ball.

\begin{proof}[Proof of Theorem \ref{thm:main1}]
Denote by $E^1 \subset H^s_{B,0}$ the eigenspace associated to the first nontrivial eigenvalue $\mu_1^{(s)}$. 
Let $E_r$ be the subspace of radially symmetric eigenfunctions, and $E_a$ the subspace generated by eigenfunctions each of which is antisymmetric with respect to some hyperplane. Let $u \in E^1$. By Corollary~\ref{cor:pol}, $u$ is foliated Schwarz symmetric with respect to some $p \in \mathbb{S}^{N-1}$. Writing
\[ u = \frac{1}{2}\left(u-u\circ \sigma_p \right)+ \frac{1}{2}(u+u \circ \sigma_p),\]
and in view of Propositions~\ref{prop:twonodaldomains} and \ref{prop:radial}, we deduce that \[E^1 = E_a \oplus E_r.\] 
Since the dimension of $E^1$ is finite, the same is true for $E_a$ and $E_r$. 
In particular, there exist at most $k \in \mathbb{N}$ linearly independent radially symmetric eigenfunctions in $E^1$. 

Assume that $E_a \neq \{0\}$ and consider the subspace $E_a$ in more detail. 
Take any $u_1 \in E_a$ which is antisymmetric with respect to $H_{e'}$ for some $e' \in \mathbb{S}^{N-1}$. 
Since the problem \eqref{eq:eigenvalueproblem0} in $B$ is rotationally invariant, we can assume, without lost of generality, that $e' = (1,0,\dots,0)$. 
Denoting by $e_j$ the $j$-th unit coordinate vector, $j=1,\dots,N$, we let $u_j$ be a rotation of $u_1$ so that $u_j \in E_a$ and it is antisymmetric with respect to $H_{e_j}$. 
By construction, the functions $u_1,\dots, u_N$ are $L^2(\Omega)$-orthogonal (and hence linearly independent) to each other.

Let us now take any $e \in \mathbb{S}^{N-1}$.
Since $e = \alpha_1 e_1 + \dots +\alpha e_N$ for some $(\alpha_1,\dots,\alpha_N) \neq (0,\dots,0)$, the function $v = \alpha_1 u_1 + \dots \alpha_N u_N$ is nonzero. 
Indeed, recalling that $u_1,\dots, u_N$ are $L^2(\Omega)$-orthogonal, we get
$$
\|v\|_{L^2(B)}^2 =
\alpha_1^2 \|u_1\|_{L^2(B)}^2
+ \dots + 
\alpha_N^2 \|u_N\|_{L^2(B)}^2
=
(\alpha_1^2 + \dots + \alpha_N^2) \|u_1\|_{L^2(B)}^2 
=
\|u_1\|_{L^2(B)}^2 > 0. 
$$
Therefore, $v \in E_a$ and it is antisymmetric with respect to $H_{e}$.  
Thanks to Proposition~\ref{prop:unique-antisym}, $v$ is the unique eigenfunction in its symmetry class, up to scaling. 
Since $v$ was arbitrary, we conclude that the functions 
$\{u_1,\dots, u_N\}$ form an orthogonal basis for $E_a$. 
The remaining properties of $u_j$ are given by Proposition~\ref{prop:twonodaldomains}. 
\end{proof}	

By means of Theorems~\ref{thm:main1} and \ref{thm:stable}, we are able to prove the antisymmetry of the first nontrivial eigenfunctions in the ball, for $s$ sufficiently close to $1$.

\begin{proof}[Proof of Theorem \ref{thm:main2}] Suppose, by contradiction, that there exists a sequence $\{s_k\}_{k \in \N} \subset (0,1)$ with $s_k \to 1^-$ as $k \to +\infty$, such that, for every $k \in \N$, there exists a radially symmetric eigenfunction $u_1^{(s_k)}$, normalized in $L^2(B)$, associated to $\mu_1^{(s_k)}$. By Theorem~\ref{thm:stable}, $u_1^{(s_k)}$ converges in $L^2(\Omega)$, up to a subsequence, to an eigenfunction $u \in H^1(B)$
associated to the first nontrivial eigenvalue $\mu_1^{(1)}$ of the local Neumann Laplacian in $B$.
But then $u$ is also radially symmetric, which is impossible, see Section~\ref{sec:intro}. 
Therefore, there exists $s^* \in (0,1)$ such that, for all $s \in (s^*,1)$, the eigenfunctions associated to $\mu_1^{(s)}$ are not radially symmetric. By Theorem \ref{thm:main1}, the eigenspace associated to $\mu_1^{(s)}$ is generated by $N$ linearly independent antisymmetric eigenfunctions.
\end{proof}

\end{document}